\documentclass[12pt, a4paper]{amsart}
\usepackage[utf8]{inputenc}
\usepackage[english]{babel}

\usepackage[margin=2.5cm]{geometry}
\usepackage[final]{graphicx}
\usepackage{epstopdf} 
\usepackage{caption}
\usepackage{subcaption}
\usepackage{comment}

\usepackage{cleveref}
\DeclareCaptionSubType[Alph]{figure}
\usepackage[T1]{fontenc}
\usepackage{mathrsfs}
\usepackage{mathtools}

\usepackage{tikz}
\usetikzlibrary{decorations.markings}
\usetikzlibrary{decorations.pathreplacing}
\usetikzlibrary{arrows,shapes,positioning,patterns}
\usetikzlibrary{knots}
\tikzstyle{none}=[inner sep=0pt]
\pgfdeclarelayer{edgelayer}
\pgfdeclarelayer{nodelayer}
\pgfsetlayers{edgelayer,nodelayer,main}
\usepackage{url}

\newcommand{\bR}{\mathbb R}

\newcommand{\bC}{\mathbb C}
\newcommand{\D}{\mathscr D}

\newcommand{\E}{\mathcal E}
\newcommand{\bZ}{\mathbb Z}

\newcommand{\ep}{\varepsilon}
\newcommand{\si}{\sigma}
\newcommand{\Si}{\Sigma}

\newcommand{\Ga}{\Gamma}

\newcommand{\sign}{\rm sgn}

\newcommand{\x}{\times}

\newcommand{\co}{\thinspace\colon}

\newcommand{\Mat}{{\rm Mat}_{n}(\bC)}

\newcommand*{\leftcrossaxis}{%
	\raisebox{-5pt}{%
		\begin{tikzpicture}[scale=0.3, every path/.style ={very thick}, 
		every node/.style={knot crossing, inner sep = 2pt}]
		\node (l1) at (-1,-1) {};
		\node (l2) at (-1,1) {};
		\node (r1) at (1,-1) {};
		\node (r2) at (1,1) {};
		\node (m) at (0,0) {};
		\draw (l1.center) -- (m) {};
		\draw (m) -- (r2.center) {};
		\draw (l2.center) -- (r1.center);
		\draw[densely dashed, red] (0,-1) -- (0,1); 
		\end{tikzpicture}}
}

\newcommand*{\rightcrossaxis}{%
	\raisebox{-5pt}{%
		\begin{tikzpicture}[scale=0.3, every path/.style ={very thick}, 
		every node/.style={knot crossing, inner sep = 2pt}]
		\node (l1) at (-1,-1) {};
		\node (l2) at (-1,1) {};
		\node (r1) at (1,-1) {};
		\node (r2) at (1,1) {};
		\node (m) at (0,0) {};
		\draw (l1.center) -- (r2.center) {};
		\draw (l2.center) -- (m) {};
		\draw (m) -- (r1.center);
		\draw[densely dashed, red] (0,-1) -- (0,1){}; 
		\end{tikzpicture}}
}

\newcommand*{\horresaxis}{%
	\raisebox{-5pt}{%
		\begin{tikzpicture}[scale=0.3, every path/.style ={very thick}, 
		every node/.style={knot crossing, inner sep = 2pt}]
		\node (l1) at (-1,-1) {};
		\node (l2) at (-1,1) {};
		\node (r1) at (1,-1) {};
		\node (r2) at (1,1) {};
		\node (m) at (0,0) {};
		\draw (l2.center) .. controls (l2.2 south east) and (r2.2 south west) .. (r2.center) {};
		\draw (l1.center) .. controls (l1.2 north east) and (r1.2 north west) .. (r1.center) {};
		\draw[densely dashed, red] (0,-1) -- (0,1){}; 
		\end{tikzpicture}}
}

\newcommand*{\vertresaxis}{%
	\raisebox{-5pt}{%
		\begin{tikzpicture}[scale=0.3, every path/.style ={very thick}, 
		every node/.style={knot crossing, inner sep = 2pt}]
		\node (l1) at (-1,-1) {};
		\node (l2) at (-1,1) {};
		\node (r1) at (1,-1) {};
		\node (r2) at (1,1) {};
		\node (m) at (0,0) {};
		\draw (l2.center) .. controls (l2.2 south east) and (l1.2 north east) .. (l1.center) {};
		\draw (r2.center) .. controls (r2.2 south west) and (r1.2 north west) .. (r1.center) {};
		\draw[densely dashed, red] (0,-1) -- (0,1){}; 
		\end{tikzpicture}}
}

\newtheorem{thm}{Theorem}[section]
\newtheorem{lemma}[thm]{Lemma}
\newtheorem{cor}[thm]{Corollary}

\newtheorem{prop}[thm]{Proposition}

\newcommand*\sm[1]{\left(\begin{smallmatrix}#1\end{smallmatrix}\right)}

\theoremstyle{definition}
\newtheorem{defn}[thm]{Definition}
\newtheorem{defns}[thm]{Definitions}

\numberwithin{equation}{section}

\begin{document}

\title{On symmetric equivalence of symmetric union diagrams}

\author{Carlo Collari}
\address{Alfredi R\'enyi Institute of Mathematics, Budapest, Hungary}
\email{carlo.collari.math@gmail.com}
\author{Paolo Lisca}
\address{Department of Mathematics, University of Pisa, ITALY} 
\email{paolo.lisca@unipi.it}
\begin{abstract} 
Eisermann and Lamm introduced a notion of symmetric equivalence among symmetric union diagrams and studied it using a refined form of the Jones polynomial. We introduced invariants of symmetric equivalence via refined versions of topological spin models and provided a partial answer to a question left open by Eisermann and Lamm. In this paper we adopt a new approach to the symmetric equivalence problem and give a complete answer to the original question left open by Eisermann and Lamm.
\end{abstract}

\maketitle

\section{Introduction}\label{s:intro}
Eisermann and Lamm introduced a notion of symmetric equivalence among symmetric union diagrams and defined a Laurent polynomial invariant under symmetric equivalence~\cite{Ei.La11}. The authors of the present paper tackled the problem of symmetric equivalence by considering a stronger version of symmetric equivalence and using topological spin models to define invariants for both types of equivalence~\cite{CL}. Here we introduce a different approach to study symmetric equivalence and, as an application, we resolve a question left open in both~\cite{Ei.La11} and~\cite{CL}. In Subsections~\ref{ss:sdse}, \ref{ss:rJp} and~\ref{ss:refspinmodels} we collect the necessary background material and in Subsection~\ref{ss:results} we state our results. 

\subsection{Symmetric diagrams and symmetric equivalences}\label{ss:sdse}

The involution of the real two-plane $\rho\co\bR^2\to\bR^2$ given by $\rho(x,y)=(-x,y)$ 
fixes the \emph{axis} $\ell=\{0\}\x\bR\subset\bR^2$ pointwise. We declare  
two diagrams $D,D'\subset\bR^2$ to be identical if one is sent to the other by an orientation-preserving diffeomorphism $h\co\bR^2\to\bR^2$ such that $h\circ\rho=\rho\circ h$ and $h(D)=D'$. An oriented link diagram $D\subset\bR^2$ is~\emph{symmetric} if $\rho(D)$ 
is obtained from $D$ by changing the orientation and switching all the crossings on the axis. 
A symmetric diagram $D$ is a~\emph{symmetric union} if $\rho$ sends each component $\widehat{D}$ of $D$ to itself in an orientation-reversing fashion, implying that $\widehat{D}$ crosses the 
axis perpendicularly in exactly two non--crossing points. Figure~\ref{f:1042} illustrates the symmetric union diagrams $D_4$ and $D'_4$, first considered by Eisermann and Lamm~\cite{Ei.La11}. 
\begin{figure}[ht]
\begin{tikzpicture}[scale = 0.5, every path/.style ={very thick}, 
every node/.style={knot crossing, inner sep = 2pt}] 
\node (l1) at (-1.5,-2.3){};
\node (l2) at (-1,-1){};
\node (l3) at (-1,1){};
\node (l4) at (-1,3){};
\node (l5) at (-2,4){};
\node (l6) at (-1,5){};
\node (r1) at (1.5,-2.3){};
\node (r2) at (1,-1){};
\node (r3) at (1,1){};
\node (r4) at (1,3){};
\node (r5) at (2,4){};
\node (r6) at (1,5){};
\node (m1) at (0,-2){};
\node (m2) at (0,0){};
\node (m3) at (0,2){};
\node (m4) at (0,4){};
\draw (l1.center) .. controls (l1.2 east) and (m1.2 south west) .. (m1){};
\draw (l1.center) .. controls (l1.8 west) and (l5.8 south west) .. (l5){};
\draw (l1) .. controls (l1.4 south) and (r1.4 south) .. (r1){};
\draw (l1) .. controls (l1.2 north) and (l2.2 south west) .. (l2.center){};
\draw (r1.center) .. controls (r1.2 west) and (m1.2 south east) .. (m1.center){};
\draw (r1.center) .. controls (r1.8 east) and (r5.8 south east) .. (r5){};
\draw (r1) .. controls (r1.2 north) and (r2.2 south east) .. (r2.center){};
\draw (m1.center) .. controls (m1.2 north west) and (l2.2 south east) .. (l2){};
\draw (m1) .. controls (m1.2 north east) and (r2.2 south west) .. (r2){};
\draw (l2) .. controls (l2.4 north west) and (l3.4 south west) .. (l3.center){};
\draw (l2.center) .. controls (l2.2 north east) and (m2.2 south west) .. (m2.center){};
\draw (m2) .. controls (m2.2 north west) and (l3.2 south east) .. (l3){};
\draw (m2.center) .. controls (m2.2 north east) and (r3.2 south west) .. (r3){};
\draw (r2.center) .. controls (r2.2 north west) and (m2.2 south east) .. (m2){};
\draw (r2) .. controls (r2.4 north east) and (r3.4 south east) .. (r3.center){};
\draw (l3.center) .. controls (l3.2 north east) and (m3.2 south west) .. (m3){};
\draw (r3.center) .. controls (r3.2 north west) and (m3.2 south east) .. (m3.center){};
\draw (l3) .. controls (l3.4 north west) and (l4.4 south west) .. (l4.center){};
\draw (r3) .. controls (r3.4 north east) and (r4.4 south east) .. (r4.center){};
\draw (l4) .. controls (l4.2 north west) and (l5.2 south east) .. (l5.center){};
\draw (r4) .. controls (r4.2 north east) and (r5.2 south west) .. (r5.center){};
\draw (l4.center) .. controls (l4.2 north east) and (m4.2 south west) .. (m4.center){};
\draw (r4.center) .. controls (r4.2 north west) and (m4.2 south east) .. (m4){};
\draw (l4) .. controls (l4.2 south east) and (m3.2 north west) .. (m3.center){};
\draw (r4) .. controls (r4.2 south west) and (m3.2 north east) .. (m3){};
\draw (l5) .. controls (l5.2 north east) and (l6.2 south west) .. (l6.center){};
\draw (r5) .. controls (r5.2 north west) and (r6.2 south east) .. (r6.center){};
\draw (l5.center) .. controls (l5.4 north west) and (l6.4 north west) .. (l6){};
\draw (r5.center) .. controls (r5.4 north east) and (r6.4 north east) .. (r6){};
\draw (m4) .. controls (m4.2 north west) and (l6.2 south east) .. (l6){};
\draw (m4.center) .. controls (m4.2 north east) and (r6.2 south west) .. (r6){};
\draw (l6.center) .. controls (l6.4 north east) and (r6.4 north west) .. (r6.center){};
\draw [dashed, red] (0,-3.5) -- (0,6){};
\begin{scope}[xshift=11cm, yshift=1.5cm]
\node (l1) at (-1,-3.2){};
\node (l2) at (-1,-1.7){};
\node (l3) at (-1,0){};
\node (l4) at (-1,1.5){};
\node (l5) at (-2,2){};
\node (l6) at (-1,3){};
\node (m1) at (0,-1){};
\node (m2) at (0,1){};
\node (r1) at (1,-3.2){};
\node (r2) at (1,-1.7){};
\node (r4) at (1,1.5){};
\node (r3) at (1,0){};
\node (r5) at (2,2){};
\node (r6) at (1,3){};
\draw (l1) .. controls (l1.4 south east) and (r1.4 south west) .. (r1){};
\draw (l1.center) .. controls (l1.2 north east) and (l2.2 south east) .. (l2){};
\draw (r1.center) .. controls (r1.2 north west) and (r2.2 south west) .. (r2){};
\draw (l1) .. controls (l1.2 north west) and (l2.2 south west) ..  (l2.center){};
\draw (r1) .. controls (r1.2 north east) and (r2.2 south east) ..(r2.center){};
\draw (l2.center) .. controls (l2.2 north east) and (m1.2 south west) .. (m1.center){};
\draw (r2.center) .. controls (r2.2 north west) and (m1.2 south east) .. (m1){};
\draw (l2) .. controls (l2.2 north west) and (l3.2 south west) ..  (l3.center){};
\draw (r2) .. controls (r2.2 north east) and (r3.2 south east) ..(r3.center){};
\draw (l3) .. controls (l3.2 south east) and (m1.2 north west) .. (m1){};
\draw (r3) .. controls (r3.2 south west) and (m1.2 north east) .. (m1.center){};
\draw (l3) .. controls (l3.2 north west) and (l4.2 south west) .. (l4.center){};
\draw (r3) .. controls (r3.2 north east) and (r4.2 south east) .. (r4.center){};
\draw (l3.center) .. controls (l3.2 north east) and (m2.2 south west) .. (m2){};
\draw (r3.center) .. controls (r3.2 north west) and (m2.2 south east) .. (m2.center){};
\draw (l1.center) .. controls (l1.8 south west) and (l5.8 south west) .. (l5){};
\draw (r1.center) .. controls (r1.8 south east) and (r5.8 south east) .. (r5){};
\draw (l5.center) .. controls (l5.2 south east) and (l4.2 west) .. (l4){};
\draw (l4) .. controls (l4.2 east) and (m2.2 north west) .. (m2.center){};
\draw (r5.center) .. controls (r5.2 south west) and (r4.2 east) .. (r4){};
\draw (r4) .. controls (r4.2 west) and (m2.2 north east) .. (m2){};
\draw (l5) .. controls (l5.2 north east) and (l6.2 south west) .. (l6.center){};
\draw (r5) .. controls (r5.2 north west) and (r6.2 south east) .. (r6.center){};
\draw (l5.center) .. controls (l5.4 north west) and (l6.4 north west) .. (l6){};
\draw (r5.center) .. controls (r5.4 north east) and (r6.4 north east) .. (r6){};
\draw (l4.center) .. controls (l4.2 north east) and (l6.2 south east) .. (l6){};
\draw (r4.center) .. controls (r4.2 north west) and (r6.2 south west) .. (r6){};
\draw (l6.center) .. controls (l6.4 north east) and (r6.4 north west) .. (r6.center){};
\draw [dashed, red] (0,-4.3) -- (0,4){};
\end{scope}
\end{tikzpicture}
\caption{The symmetric union diagrams $D_4$ (left) and $D^\prime_4$ (right)}
\label{f:1042}
\end{figure}
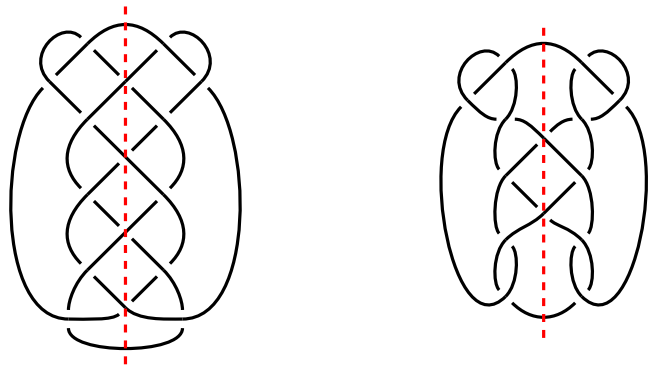
Following~\cite{Ei.La11}, we define a~\emph{symmetric Reidemeister move off the axis} as an ordinary Reidemeister move on a symmetric diagram carried out, away from the axis $\ell$, together with its mirror-symmetric counterpart with respect to $\ell$. A~\emph{symmetric Reidemeister move on the axis} is one of the moves illustrated in Figure~\ref{f:sRm}. 
Eisermann and Lamm consider also two extra moves S1($\pm$) and S2(v), some of which are illustrated in Figure~\ref{f:extrasRm}. 
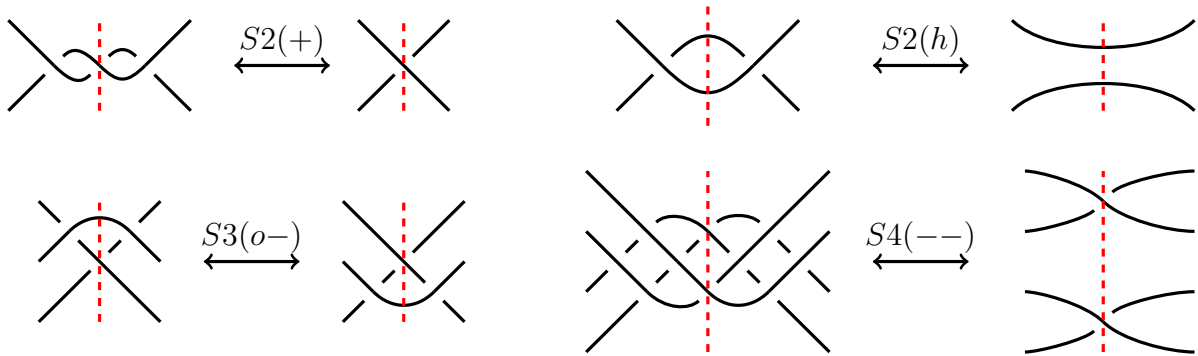
\begin{figure}[ht]
\centering
\begin{tikzpicture}[scale=0.4, every path/.style ={very thick}, 
every node/.style={knot crossing, inner sep = 3pt}] 
\node (l1) at (-2,-2) {};
\node (l2) at (-2,0) {};
\node (l3) at (-1,1) {};
\node (l4) at (-2,2) {};
\node (r1) at (2,-2) {};
\node (r2) at (2,0) {};
\node (r3) at (1,1) {};
\node (r4) at (2,2) {};
\node (m) at (0,0) {};
\draw (l1.center) -- (m) {};
\draw (m) -- (r3) {};
\draw (l3) -- (r1.center);
\draw (l4.center) -- (l3) {};
\draw (r3) -- (r4.center) {};
\draw (l2.center) .. controls (l2.2 north east) and (l3.2 south west) .. (l3.center);
\draw (l3.center) .. controls (l3.2 north east) and (r3.2 north west) .. (r3.center);
\draw (r3.center) .. controls (r3.2 south east) and (r2.2 north west) .. (r2.center);
\draw[dashed, red] (0,-2) -- (0,2); 
\begin{scope}[xshift=5cm]
\node (l) at (-2,0) {};
\node (r) at (2,0) {};
\node at (0,0.8) {$S3(o-)$};
\draw[<->] (l) -- (r);
\end{scope}
\begin{scope}[xshift=10cm]
\node (l1) at (-2,-2) {};
\node (l2) at (-1,-1) {};
\node (l3) at (-2,0) {};
\node (l4) at (-2,2) {};
\node (r1) at (2,-2) {};
\node (r2) at (1,-1) {};
\node (r3) at (2,0) {};
\node (r4) at (2,2) {};
\node (m) at (0,0) {};
\draw (l1.center) -- (l2) {};
\draw (m) -- (r4.center) {};
\draw (r2) -- (r1.center);
\draw (l2) -- (m) {};
\draw (l4.center) -- (r2) {};
\draw (l3.center) .. controls (l3.2 south east) and (l2.2 north west) .. (l2.center);
\draw (l2.center) .. controls (l2.2 south east) and (r2.2 south west) .. (r2.center);
\draw (r2.center) .. controls (r2.2 north east) and (r3.2 south west) .. (r3.center);
\draw[dashed, red] (0,-2) -- (0,2); 	
\end{scope}
\begin{scope}[yshift=5cm]
\node (l1) at (-3,0) {};
\node (l2) at (-1.5,1.5) {};
\node (l3) at (-3,3) {};
\node (m) at (0,1.5) {};
\node (r1) at (3,0) {};
\node (r2) at (1.5,1.5) {};
\node (r3) at (3,3) {};
\draw (l1.center) .. controls (l1.2 north east) and (l2.2 south west) .. (l2);
\draw (l2.center) .. controls (l2.2 north west) and (l3.2 south east) .. (l3.center);
\draw (r1.center) .. controls (r1.2 north west) and (r2.2 south east) .. (r2);
\draw (r2.center) .. controls (r2.2 north east) and (r3.2 south west) .. (r3.center);
\draw (l2.center) .. controls (l2.2 south east) and (m.2 south west) .. (m);
\draw (l2) .. controls (l2.2 north east) and (m.2 north west) .. (m.center);
\draw (r2.center) .. controls (r2.2 south west) and (m.2 south east) .. (m.center);
\draw (r2) .. controls (r2.2 north west) and (m.2 north east) .. (m);
\draw[dashed, red] (0,0) -- (0,3); 
\end{scope}
\begin{scope}[xshift=6cm,yshift=5cm]
\node (l) at (-2,1.5) {};
\node (r) at (2,1.5) {};
\node at (0,2.3) {$S2(+)$};
\draw[<->] (l) -- (r);
\end{scope}
\begin{scope}[xshift=10cm,yshift=5cm]
\node (l1) at (-1.5,0) {};
\node (l3) at (-1.5,3) {};
\node (m) at (0,1.5) {};
\node (r1) at (1.5,0) {};
\node (r3) at (1.5,3) {};
\draw (l1.center) -- (m);
\draw (l3.center) -- (r1.center);
\draw (m) -- (r3.center);
\draw[dashed, red] (0,0) -- (0,3); 
\end{scope}
\begin{scope}[xshift=20cm]
\node (l4-3) at (-4,-3) {};
\node (l4-1) at (-4,-1) {};
\node (l41) at (-4,1) {};
\node (l43) at (-4,3) {};
\node (l30) at (-3,0) {};
\node (l2-1) at (-2,-1) {};
\node (l21) at (-2,1) {};
\node (l10) at (-1,0) {};
\node (m-2) at (0,-1) {};
\node (m2) at (0,1) {};
\node (r4-3) at (4,-3) {};
\node (r4-1) at (4,-1) {};
\node (r41) at (4,1) {};
\node (r43) at (4,3) {};
\node (r30) at (3,0) {};
\node (r2-1) at (2,-1) {};
\node (r21) at (2,1) {};
\node (r10) at (1,0) {};
\draw (l4-3.center) -- (l2-1){};
\draw (l4-1.center) -- (l30){};
\draw (l41.center) -- (l2-1.center){};
\draw (l43.center) -- (l10.center){};
\draw (l30) -- (l21){};
\draw (l2-1) -- (l10){};
\draw (l2-1.center) .. controls (l2-1.2 south east) and (m-2.2 south west) .. (m-2){};
\draw (l10.center) .. controls (l10.2 south east) and (m-2.2 north west) .. (m-2.center){};
\draw (l21) .. controls (l21.2 north east) and (m2.2 north west) .. (m2.center){};
\draw (l10) .. controls (l10.2 north east) and (m2.2 south west) .. (m2){};
\draw (m2.center) .. controls (m2.2 south east) and (r10.2 north west) .. (r10){};
\draw (m2) .. controls (m2.2 north east) and (r21.2 north west) .. (r21){};
\draw (m-2.center) .. controls (m-2.2 south east) and (r2-1.2 south west) .. (r2-1.center){};
\draw (m-2) .. controls (m-2.2 north east) and (r10.2 south west) .. (r10.center) {};
\draw (r10.center) -- (r43.center){};
\draw (r10) -- (r2-1){};
\draw (r2-1.center) -- (r41.center){};
\draw (r21) -- (r30){};
\draw (r30) -- (r4-1.center){};
\draw (r2-1) -- (r4-3.center){};
\draw[dashed, red] (0,-3) -- (0,3){};
\end{scope}
\begin{scope}[xshift=27cm]
\node (l) at (-2,0) {};
\node (r) at (2,0) {};
\node at (0,0.8) {$S4(--)$};
\draw[<->] (l) -- (r);
\end{scope}
\begin{scope}[xshift=33cm]
\node (l3-3) at (-3,-3){};
\node (l3-1) at (-3,-1){};
\node (l31) at (-3,1){};
\node (l33) at (-3,3){};
\node (m-2) at (0,-2){};
\node (m2) at (0,2){};
\node (r3-3) at (3,-3){};
\node (r3-1) at (3,-1){};
\node (r31) at (3,1){};
\node (r33) at (3,3){};
\draw (l33) .. controls (l33.2 east) and (m2.2 north west) .. (m2.center){}; 
\draw (l31) .. controls (l31.2 east) and (m2.2 south west) .. (m2){}; 
\draw (l3-1) .. controls (l3-1.2 east) and (m-2.2 north west) .. (m-2.center){}; 
\draw (l3-3) .. controls (l3-3.2 east) and (m-2.2 south west) .. (m-2){}; 
\draw (m-2.center) .. controls (m-2.2 south east) and (r3-3.2 west) .. (r3-3.center){};
\draw (m-2) .. controls (m-2.2 north east) and (r3-1.2 west) .. (r3-1.center){};
\draw (m2.center) .. controls (m2.2 south east) and (r31.2 west) .. (r31.center){};
\draw (m2) .. controls (m2.2 north east) and (r33.2 west) .. (r33.center){};
\draw[dashed, red] (0,-3) -- (0,3);
\end{scope}
\begin{scope}[xshift=20cm,yshift=5cm]
\node (l1) at (-3,0) {};
\node (l2) at (-1.5,1.5) {};
\node (l3) at (-3,3) {};
\node (r1) at (3,0) {};
\node (r2) at (1.5,1.5) {};
\node (r3) at (3,3) {};
\draw (l1.center) .. controls (l1.2 north east) and (l2.2 south west) .. (l2);
\draw (l2.center) .. controls (l2.2 north west) and (l3.2 south east) .. (l3.center);
\draw (r1.center) .. controls (r1.2 north west) and (r2.2 south east) .. (r2);
\draw (r2.center) .. controls (r2.2 north east) and (r3.2 south west) .. (r3.center);
\draw (l2.center) .. controls (l2.4 south east) and (r2.4 south west) .. (r2.center);
\draw (l2) .. controls (l2.4 north east) and (r2.4 north west) .. (r2);
\draw[dashed, red] (0,-0.5) -- (0,3.5); 
\end{scope}
\begin{scope}[xshift=27cm,yshift=5cm]
\node (l) at (-2,1.5) {};
\node (r) at (2,1.5) {};
\node at (0,2.3) {$S2(h)$};
\draw[<->] (l) -- (r);
\end{scope}
\begin{scope}[xshift=33cm,yshift=5cm]
\node (l1) at (-3,0) {};
\node (l3) at (-3,3) {};
\node (r1) at (3,0) {};
\node (r3) at (3,3) {};
\draw (l1.center) .. controls (l1.4 north east) and (r1.4 north west) .. (r1.center);
\draw (l3.center) .. controls (l3.4 south east) and (r3.4 south west) .. (r3.center);
\draw[dashed, red] (0,0) -- (0,3); 
\end{scope}
\end{tikzpicture}
\caption{Symmetric Reidemeister moves on the axis}
\label{f:sRm}
\end{figure} 
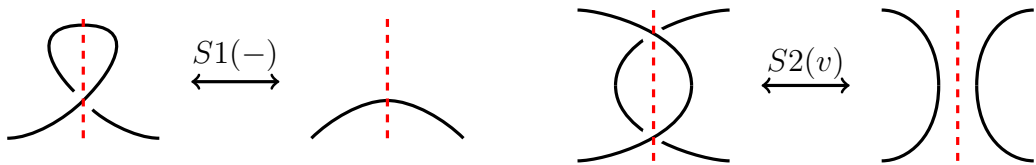
\begin{figure}[ht]
	\centering
	\pgfdeclarelayer{layer1}
\pgfdeclarelayer{layer2}
\pgfdeclarelayer{layer3}
\pgfsetlayers{layer1,layer2,layer3}
\begin{tikzpicture}[scale=0.5, every path/.style ={very thick}, 
every node/.style={knot crossing, inner sep = 3pt}] 
\begin{pgfonlayer}{layer1}
\begin{scope}[yshift=0.6cm]
\node (l1) at (-2,0) {};
\node (r1) at (2,0) {};
\node (c1) at (0,1) {};
\node (c2) at (0,3) {};
\draw (l1.center) .. controls (l1.2 east) and (c1.2 south west) .. (c1.center);
\draw (c1.center) .. controls (c1.4 north east) and (c2.4 east) .. (c2.center);
\draw (c2.center) .. controls (c2.4 west) and (c1.4 north west) .. (c1);
\draw (c1) .. controls (c1.2 south east) and (r1.2 west) .. (r1.center);
\draw[dashed, red] (0,0) -- (0,3.3); 
\end{scope}
\begin{scope}[xshift=4cm,yshift=0.6cm]
\node (l) at (-1.5,1.5) {};
\node (r) at (1.5,1.5) {};
\node at (0,2.2) {$S1(-)$};
\draw[<->] (l) -- (r);
\end{scope}
\begin{scope}[xshift=8cm,yshift=0.6cm]
\node (l1) at (-2,0) {};
\node (r1) at (2,0) {};
\node (c1) at (0,1) {};
\draw (l1.center) .. controls (l1.2 north east) and (c1.2 west) .. (c1.center);
\draw (c1.center) .. controls (c1.2 east) and (r1.2 north west) .. (r1.center);
\draw[dashed, red] (0,0) -- (0,3.3); 
\end{scope}
\begin{scope}[xshift=15cm]
\node (l1) at (-2,0) {};
\node (l2) at (-1,2) {};
\node (l3) at (-2,4) {};
\node (r1) at (2,0) {};
\node (r2) at (1,2) {};
\node (r3) at (2,4) {};
\begin{pgfonlayer}{layer2}
\node[circle,color=white,fill=white,radius=5pt] (c1) at (0,0.7) {};
\node[circle,color=white,fill=white,radius=5pt] (c2) at (0,3.3) {};
\end{pgfonlayer}
\begin{pgfonlayer}{layer3}
\draw (l1.center) .. controls (l1.2 east) and (r2.4 south) .. (r2.center);
\draw (r2.center) .. controls (r2.4 north) and (l3.2 east) .. (l3.center);
\draw[dashed, red] (0,0) -- (0,4); 
\end{pgfonlayer}
\draw (r1.center) .. controls (r1.2 west) and (l2.4 south) .. (l2.center);
\draw (l2.center) .. controls (l2.4 north) and (r3.2 west) .. (r3.center);
\end{scope}
\begin{scope}[xshift=19cm]
\node (l) at (-1.5,2) {};
\node (r) at (1.5,2) {};
\node at (0,2.6) {$S2(v)$};
\draw[<->] (l) -- (r);
\end{scope}
\begin{scope}[xshift=23cm]
\node (l1) at (-2,0) {};
\node (l2) at (-0.5,2) {};
\node (l3) at (-2,4) {};
\node (r1) at (2,0) {};
\node (r2) at (0.5,2) {};
\node (r3) at (2,4) {};
\draw (l1.center) .. controls (l1.2 east) and (l2.4 south) .. (l2.center);
\draw (l2.center) .. controls (l2.4 north) and (l3.2 east) .. (l3.center);
\draw (r1.center) .. controls (r1.2 west) and (r2.4 south) .. (r2.center);
\draw (r2.center) .. controls (r2.4 north) and (r3.2 west) .. (r3.center);
\draw[dashed, red] (0,0) -- (0,4); 
\end{scope}
\end{pgfonlayer}
\end{tikzpicture}
	\caption{Moves $S1(-)$ and $S2(v)$}
	\label{f:extrasRm}
\end{figure} 
It is understood all of these moves admit variants obtained by turning the corresponding pictures upside down, mirroring or rotating them around the axis (cf.~\cite[\S 2.3]{Ei.La11}). 

\begin{defns}\label{d:SE}
Two oriented, symmetric diagrams which can be obtained from each other via a finite sequence of symmetric Reidemester moves on and off the axis (sR-moves) and S1-moves will be called~\emph{symmetrically equivalent}. If they can be obtained from each other using sR-moves, S1- and S2(v)-moves, we will say that the diagrams are~\emph{weakly symmetrically equivalent}. 
\end{defns}

\subsection{Eisermann and Lamm's results}\label{ss:rJp}
Eisermann and Lamm~\cite{Ei.La11} showed that there exists an infinite family of pairs $(D_n, D^\prime_n)$ of symmetric union $2$-bridge knot diagrams  such that $D_n$ and $D^\prime_n$ are Reidemeister equivalent but not weakly symmetrically equivalent for $n = 3$ and $n\geq 5$.
$D_4$ and $D'_4$ are the diagrams of Figure~\ref{f:1042}.
Eisermann and Lamm established their result using an invariant of weak symmetric equivalence defined as follows. Let $\vec{\D}$ denote the set of oriented planar link diagrams $D\subset\bR^2$ transverse to the axis $\ell =\{0\}\x\bR$. Let $\bZ(s^{1/2},t^{1/2})$ be the quotient field of the ring of Laurent polynomials with integer coefficients in the variables $s^{1/2}$ and $t^{1/2}$. Eisermann and Lamm~\cite{Ei.La11} define a map $W\co\vec{\D}\to\bZ(s^{1/2},t^{1/2})$ such that $W(D)=W(D')$ if $D$ and $D'$ are weakly symmetrically equivalent. By~\cite[Proposition~5.6]{Ei.La11}, if $D\in\vec{\D}$ represents a link $L$ and has no crossings on the axis then 
\begin{equation}\label{e:nocrossings}
W (D) = \left(\frac{s^{1/2}+s^{-1/2}}{t^{1/2}+t^{-1/2}}\right)^{n-1} V_L(t),
\end{equation} 
where $V_L(t)$ is the Jones-polynomial of the link $L$, normalized so that on the $n$-component unlink it takes the value $(-t^{1/2}-t^{-1/2})^{n-1}$. Moreover, if $D$ has crossings on the axis, then the following skein-like recursion formulas hold: 
\begin{align}
W\left(\rightcrossaxis{}\right) & = -s^{-1/2} W\left(\horresaxis{}\right) - s^{-1} W\left(\vertresaxis\right)\label{e:recursionright}\\
W\left(\leftcrossaxis{}\right) & = -s^{1/2} W\left(\horresaxis{}\right) - s W\left(\vertresaxis\right)\label{e:recursionleft}
\end{align}
It turns out~\cite[Proposition~1.8]{Ei.La11} that when $D$ is a symmetric union knot diagram, then $W(D)$ is an honest Laurent polynomial that we shall call the {\em refined Jones polynomial}. The diagrams $D_4$ and $D'_4$ have the same refined Jones polynomial, so the question of their weak symmetric equivalence was left unanswered in~\cite{Ei.La11}. 

\subsection{Invariants from topological spin models}\label{ss:refspinmodels}
The theory of topological spin models for links in $S^3$ was introduced in~\cite{Jo89}. Here we follow the reformulation used in~\cite{CL}, to which we refer the reader for further details. Fix an integer $n\geq 2$, denote by $\Mat$ the space of square $n\x n$ complex matrices, and let $d\in\{\pm\sqrt{n}\}$. Given a symmetric, complex matrix $W^+\in \Mat$ with non-zero entries, let $W^-\in \Mat$ be the matrix uniquely determined by the equation 
\begin{equation}\label{e:type-II}
W^+\circ W^- = J,
\end{equation}
where $\circ$ is the Hadamard, i.e.~entry-wise, product and $J$ is the all-$1$ matrix.
Define, for each matrix $X\in \Mat$ with non-zero entries and $a,b\in \{ 1,...,n\}$, 
the vector $Y^X_{ab}\in\bC^n$ by setting 
\[
Y^X_{ab}(x) := \frac {X(x,a)}{X(x,b)}\in\bC,\quad x\in \{ 1,...,n\}.  
\]
Then, the pair $M=(W^+,d)$ is a {\em spin model} if the following equations hold:
\begin{equation}\label{e:type-III}
W^+ Y^{W^+}_{ab} = d W^-(a,b) Y^{W^+}_{ab}\quad\text{for every $a,b\in \{ 1,...,n\}$}.
\end{equation} 

The following definition was introduced in~\cite[Remark~1.3]{CL}.

\begin{defn}\label{d:refined-spin-model}
	A {\em Potts-refined spin model} is a triple $\widehat M = (W^+,V^+,d)$ such that: 
	\begin{itemize} 
		\item
		$M=(W^+,d)$ is a spin model; 
		\item
		$V^+ = (-\xi^{-3}) I + \xi (J-I)$, where $\xi$ is one of the four complex numbers such that $d=-\xi^2-\xi^{-2}$.
	\end{itemize}
\end{defn} 

Let $\widehat M=(W^+,V^+,d)$ be a a Potts-refined spin model, $D$ a symmetric union diagram  
and $c$ a chequerboard colouring of $\bR^2\setminus D$. 
Let $\Ga_D$ be the planar, signed medial graph associated to the black regions of $c$. 
Let $\Ga_D^0$, $\Ga_D^1$ be the sets of vertices, respectively edges of $\Ga_D$ and let $N=|\Ga_D^0|$.
Given $e\in\Ga^1$, we denote by $v_e$ and $w_e$ (in any order) the vertices of $e$.
The set $\Ga_D^1$ contains the set $\Ga_\ell^1$ of edges corresponding to crossings on the axis.  
Let $V^- = (-\xi^3)I + \xi^{-1} (J-I)$, and define the {\em partition function} $Z_{\widehat M}(D,c)$ by the formula
\[
Z_{\widehat M}(D,c) := d^{-N} \sum_{\si\co\Ga^0_D\to \{1,...,n\}} 
\prod_{e\in\Ga_\ell^1} V^{s(e)}(\si(v_e),\si(w_e))\prod_{e\in\Ga_D^1\setminus \Ga_\ell^1} W^{s(e)}(\si(v_e),\si(w_e)), 
\]
where $s(e)\in\{+,-\}$ is the sign of the edge $e$, and the 
{\em normalized partition function} $I_{\widehat M}(D,c)$ by 
\[
I_{\widehat M}(D,c) :=  (-\xi)^{p_\ell(D)-n_\ell(D)} Z_{\widehat M}(D,c),
\]
where $p_\ell(D)$ and $n_\ell(D)$ denote, respectively, the numbers of positive and negative 
crossings on the axis. When $D$ is not connected 
$Z_{\widehat M}(D,c)$ and $I_{\widehat M}(D,c)$ are defined as the product of the values of $Z_{\widehat M}$ 
and, respectively, $I_{\widehat M}$ on its connected components with the induced colourings. 
It turns out~\cite{CL} that the complex number $I_{\widehat M}(D,c)$ is independent 
of the choice of $c$, so we can write more simply $I_{\widehat M}(D)$. 
Moreover, by a special case of~\cite[Theorem~1.3]{CL}, if $D$ and $D'$ are oriented, weakly symmetrically equivalent symmetric union diagrams, then
\begin{equation}\label{e:se-invariance}
I_{\widehat M}(D) = I_{\widehat M}(D').
\end{equation}
In~\cite[Subsection~4.2]{CL} we showed that, for a suitable choice of $M$, the invariant $I_{\widehat M}$ 
defined above can distinguish, up to weak symmetric equivalence, infinitely many Reidemeister equivalent 
symmetric union diagrams. We also showed~\cite[Subsection~4.2]{CL} that more general invariants can 
distinguish the diagrams $D_4$ and $D'_4$ up to symmetric equivalence, but we were unable to use  
invariants coming from spin models to rule out that the diagrams $D_4$ and $D'_4$ of Figure~\ref{f:1042} are 
{\em weakly} symmetrically equivalent. 

\subsection{Statements of results}\label{ss:results} 
Given a symmetric union diagram $D$ and an integer $h\in\bZ$, define a new symmetric union diagram $D(h)$ by replacing each crossing on the axis with $|h|$ consecutive crossings, having the same or opposite type depending on the sign of $h$. The precise convention is specified in Figure~\ref{f:D(h)}, where a number $m=\pm h$ inside a box denotes a sequence of $|m|$ consecutive half-twists on the axis, each of sign equal to $\sign(m)$, the sign of $m$. 
\begin{figure}[ht]
	\centering	
	\begin{tikzpicture}[scale=0.55, every path/.style ={very thick}, 
every node/.style={knot crossing, inner sep = 3pt}] 
\begin{scope}[yshift=1cm]
\node at (0,-1) {$D$};
\node (l1) at (-1.5,0) {};
\node (l3) at (-1.5,3) {};
\node (m) at (0,1.5) {};
\node (r1) at (1.5,0) {};
\node (r3) at (1.5,3) {};
\draw (l1.center) -- (r3.center);
\draw (l3.center) -- (m);
\draw (m) -- (r1.center);
\draw[dashed, red] (0,0) -- (0,3); 
\end{scope}
\begin{scope}[xshift=3.5cm, yshift=1cm]
\node (l) at (-1.5,1.5) {};
\node (r) at (1.5,1.5) {};
\draw[->] (l) -- (r);
\end{scope}
\begin{scope}[xshift=7cm, yshift=1cm]
\node at (0,-1) {$D(h)$};
\node (l1) at (-1.5,0) {};
\node (l3) at (-1.5,3) {};
\node (m) at (0,1.5) {};
\node (r1) at (1.5,0) {};
\node (r3) at (1.5,3) {};
\draw (l1.center) -- (m);
\draw (l3.center) -- (r1.center);
\draw (m) -- (r3.center);
\draw[dashed, red] (0,0) -- (0,3); 
\draw[fill=white] (-0.9,0.75) rectangle (0.9,2.25);
\node at (0,1.5) {$h$}; 
\end{scope}
\begin{scope}[xshift=13cm, yshift=1cm]
\node at (0,-1) {$D$};
\node (l1) at (-1.5,0) {};
\node (l3) at (-1.5,3) {};
\node (m) at (0,1.5) {};
\node (r1) at (1.5,0) {};
\node (r3) at (1.5,3) {};
\draw (l1.center) -- (m);
\draw (l3.center) -- (r1.center);
\draw (m) -- (r3.center);
\draw[dashed, red] (0,0) -- (0,3); 
\end{scope}
\begin{scope}[xshift=16.5cm, yshift=1cm]
\node (l) at (-1.5,1.5) {};
\node (r) at (1.5,1.5) {};
\draw[->] (l) -- (r);
\end{scope}
\begin{scope}[xshift=20cm, yshift=1cm]
\node at (0,-1) {$D(h)$};
\node (l1) at (-1.5,0) {};
\node (l3) at (-1.5,3) {};
\node (m) at (0,1.5) {};
\node (r1) at (1.5,0) {};
\node (r3) at (1.5,3) {};
\draw (l1.center) -- (m);
\draw (l3.center) -- (r1.center);
\draw (m) -- (r3.center);
\draw[dashed, red] (0,0) -- (0,3); 
\draw[fill=white] (-0.9,0.75) rectangle (0.9,2.25);
\node at (0,1.5) {$-h$}; 
\end{scope}
\end{tikzpicture}
	\vspace*{0.7cm}
	\caption{Definition of $D(h)$, $h\in\bZ$}
	\label{f:D(h)}
\end{figure}
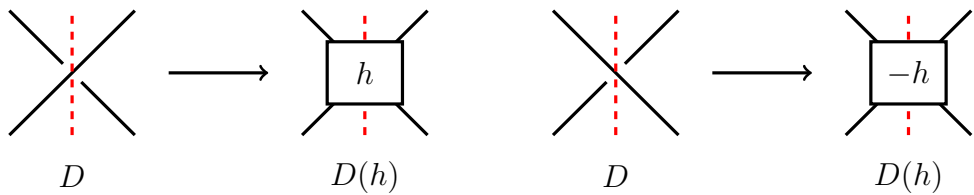  

The following theorem is established in Section~\ref{s:invariance}. 
\begin{thm}\label{t:invariance}
Let $D$ and $D^\prime$ be two symmetric diagrams. If $D$ and $D^\prime$ are (weakly) symmetrically equivalent, then $D(h)$ and $D^\prime(h)$ are (weakly) symmetrically equivalent for each $h\in \mathbb{Z}$.
\end{thm}
Clearly, if for any integer $h\in\bZ$ the diagrams $D(h)$ and $D'(h)$ can be shown to be (weakly) symmetrically inequivalent, it follows from Theorem~\ref{t:invariance} that $D$ and $D'$ cannot be (weakly) symmetrically equivalent. It is therefore natural to ask whether the weak symmetric equivalence of $D_4$ and $D'_4$ could be decided by showing that $D_4(h)$ and $D'_4(h)$ have different refined Jones polynomials or different Potts-refined spin model invariants. It turns out that this is impossible: in Section~\ref{s:negative} we show that, for any $h\in\bZ$, the diagrams $D_4(h)$ and $D'_4(h)$ have the same refined Jones polynomial and Potts-refined spin model invariants. 
Nevertheless, in Section~\ref{s:applications} we use Theorem~\ref{t:invariance} to prove the following. 
\begin{thm}\label{t:main}
The Reidemeister equivalent symmetric union diagrams $D_4$ and $D'_4$ are not weakly symmetrically equivalent. 
\end{thm}
Notice that Theorem~\ref{t:main} resolves the question left open in~\cite{Ei.La11, CL} about the weak symmetric equivalence of the diagrams $D_4$ and $D'_4$. The proof of Theorem~\ref{t:main} is based on the simple fact that if two symmetric union diagrams $D$ and $D'$ are symmetrically equivalent then they are, in particular, Reidemeister equivalent and therefore represent the same link in $S^3$. Thus, to prove  Theorem~\ref{t:main} it suffices to show that the knots $K$ and $K'$, represented respectively by the diagrams $D_4(2)$ and $D'_4(2)$, are distinct. This can be accomplished in a number of ways. We sketch a few, and provide the details of a computation showing that $K$ and $K'$ have different torsion numbers. 

The paper is organized as follows. In Section~\ref{s:invariance} we prove Theorem~\ref{t:invariance}. In Section~\ref{s:negative} we show that $D_4(h)$ and $D_4'(h)$ have the same refined Jones polynomial and 
Potts-refined spin model invariants. In Section~\ref{s:applications} we prove Theorem~\ref{t:main}. 


\section{Proof of Theorem~\ref{t:invariance}}\label{s:invariance}
The following Lemmas~\ref{lemma:S4mn} and~\ref{lemma:S2n} deal with generalizations of, respectively, the $S4$-move and the $S2$-move. The lemmas play a key r\^ole in the proof of Theorem~\ref{t:invariance}.

\begin{lemma}\label{lemma:S4mn}
Suppose that the symmetric diagrams $D$ and $D^\prime$ differ by the~$S4(m, n)$-move defined in Figure \ref{fig:S4generale}. Then, $D$ and $D^\prime$ are connected by a sequence of symmetric Reidemeister moves off the axis and $S4$-moves. In particular, $D$ and $D'$ 
are symmetrically equivalent. 
\end{lemma}
\begin{figure}[ht]
\begin{tikzpicture}[thick, scale =.6]

\draw[red, very thick, dashed] (0,3)--(0,-3);

\draw (3.5,-1) .. controls +(-1.5,0) and +(.5,0) .. (.5,2.5) .. controls +(-.5,0) and +(.5,0) .. (-.5,1) .. controls +(-.5,0) and +(1.5,0) .. (-3.5,-2.5);
\draw (3.5,-2.5)  .. controls +(-1.5,0) and +(.5,0) .. (.5,1) .. controls +(-.5,0) and +(.5,0) .. (-.5,2.5) .. controls +(-.5,0) and +(1.5,0) .. (-3.5,-1);
\draw[white, fill] (.6,1.25) rectangle (-.6,2.25);
\draw[] (.6,1.25) rectangle (-.6,2.25);
\node at (0,1.75) {$m$};

\pgfsetlinewidth{5*\pgflinewidth}
\draw[white] (3.5,1).. controls +(-1.5,0) and +(.5,0) .. (.5,-2.5) .. controls +(-.5,0) and +(.5,0) .. (-.5,-1) .. controls +(-.5,0) and +(1.5,0) .. (-3.5,2.5);
\draw[white] (3.5,2.5) .. controls +(-1.5,0) and +(.5,0) ..  (.5,-1) .. controls +(-.5,0) and +(.5,0) .. (-.5,-2.5) .. controls +(-.5,0) and +(1.5,0) .. (-3.5,1);
\pgfsetlinewidth{.2*\pgflinewidth}
\draw (3.5,1).. controls +(-1.5,0) and +(.5,0) .. (.5,-2.5) .. controls +(-.5,0) and +(.5,0) .. (-.5,-1) .. controls +(-.5,0) and +(1.5,0) .. (-3.5,2.5);
\draw (3.5,2.5) .. controls +(-1.5,0) and +(.5,0) ..  (.5,-1) .. controls +(-.5,0) and +(.5,0) .. (-.5,-2.5) .. controls +(-.5,0) and +(1.5,0) .. (-3.5,1);
\draw[white, fill] (.6,-1.25) rectangle (-.6,-2.25);
\draw[] (.6,-1.25) rectangle (-.6,-2.25);
\node at (0,-1.75) {$n$};

\begin{scope}[shift={+(12,0)}]

\draw[red, very thick, dashed] (0,3)--(0,-3);

\draw (3.5,2.5) -- (.5,2.5) .. controls +(-.5,0) and +(.5,0) .. (-.5,1)--(-3.5,1);
\draw (3.5,1) --  (.5,1) .. controls +(-.5,0) and +(.5,0) .. (-.5,2.5)--(-3.5,2.5);
\draw[white, fill] (.6,1.25) rectangle (-.6,2.25);
\draw[] (.6,1.25) rectangle (-.6,2.25);
\node at (0,1.75) {$n$};

\draw (3.5,-2.5) -- (.5,-2.5) .. controls +(-.5,0) and +(.5,0) .. (-.5,-1)--(-3.5,-1);
\draw (3.5,-1) --  (.5,-1) .. controls +(-.5,0) and +(.5,0) .. (-.5,-2.5)--(-3.5,-2.5);
\draw[white, fill] (.6,-1.25) rectangle (-.6,-2.25);
\draw[] (.6,-1.25) rectangle (-.6,-2.25);
\node at (0,-1.75) {$m$};
\end{scope}

\node at (6,.5) {$S4(m, n)$};
\draw[latex-latex] (4.5,0) -- (7.5,0);
\end{tikzpicture}
\caption{Definition of the $S4(m, n)$-move, where $m$ and $n$ are integers denoting the number of crossings on the axis with the appropriate signs.}
\label{fig:S4generale}
\end{figure}
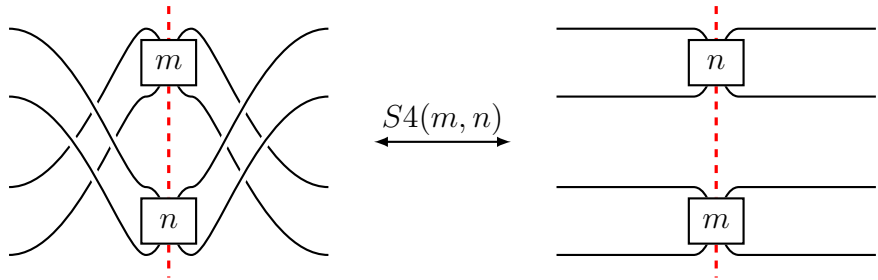

\begin{proof}
Since when $n=0$ or $m=0$ an S4($m$,$n$)-move reduces to a symmetric pair of second Reidemeister moves off the axis, we may assume without loss of generality that $mn \ne 0$. 
We suppose first that $m$ and $n$ are both positive and we establish the statement by induction on $m$ and $n$. The basis of the induction holds because an $S4(1, 1)$-move is just an ordinary $S4$-move. 
Assume that the statement holds for $S4(h, k)$-moves with $1 \leq h < m$ and $1\leq  k < n$. The inductive step is established by proving the statement for $S4(m, k)$-moves and $S4(h, n)$-moves. 
Figure~\ref{fig:S4ind1} shows that an $S4(m, k)$-move can be decomposed into a sequence of symmetric Reidemeister moves and $S4(h,k)$-moves with $1\leq h<m$. 
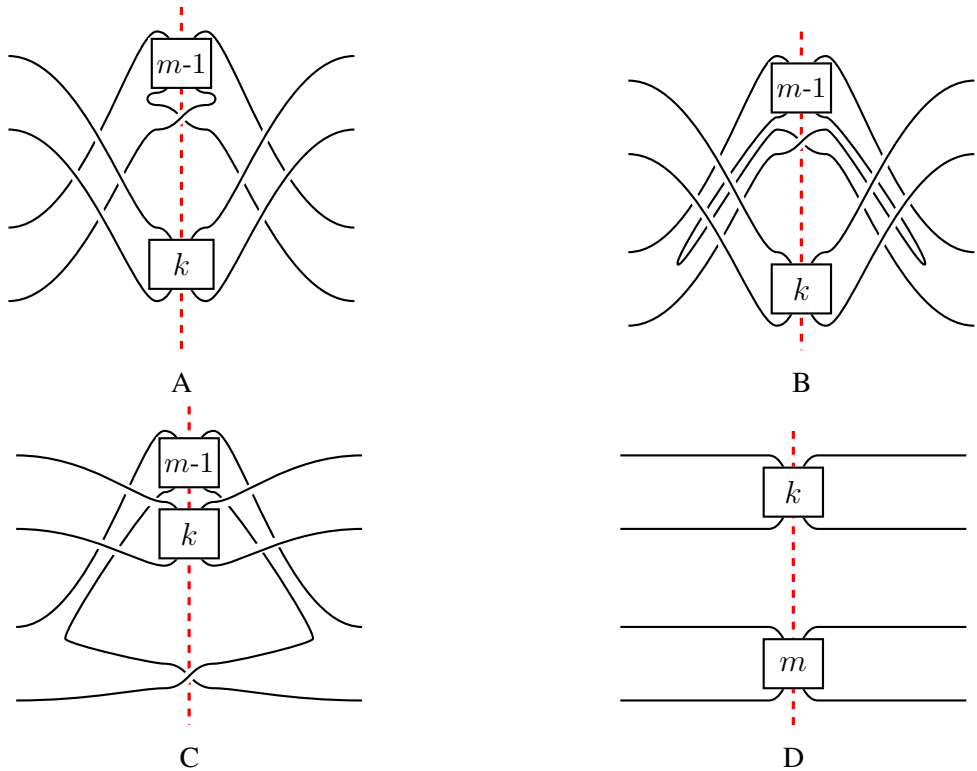
\begin{figure}[ht]
    \centering
    \begin{subfigure}[t]{0.49\textwidth}
    \centering
\begin{tikzpicture}[thick, scale = .65]
\draw[red, very thick, dashed] (0,3.5)--(0,-3.5);

\draw (3.5,-2.5)  .. controls +(-1.5,0) and +(.5,0) .. (.5,1) .. controls +(-.5,0) and +(.5,0) .. (-.5,1.5).. controls +(-.25,0) and +(-.25,0) .. (-.5,1.75) .. controls +(.5,0) and +(.5,0) .. (-.5,3) .. controls +(-.5,0) and +(1.5,0) .. (-3.5,-1);

\pgfsetlinewidth{5*\pgflinewidth}
\draw[white] (3.5,-1) .. controls +(-1.5,0) and +(.5,0) .. (.5,3)  .. controls +(-.5,0) and +(-.5,0) .. (.5,1.75) .. controls +(.25,0) and +(.25,0) .. (.5,1.5)  .. controls +(-.5,0) and +(.5,0) .. (-.5,1) .. controls +(-.5,0) and +(1.5,0) .. (-3.5,-2.5);
\pgfsetlinewidth{.2*\pgflinewidth}

\draw (3.5,-1) .. controls +(-1.5,0) and +(.5,0) .. (.5,3)  .. controls +(-.5,0) and +(-.5,0) .. (.5,1.75) .. controls +(.25,0) and +(.25,0) .. (.5,1.5)  .. controls +(-.5,0) and +(.5,0) .. (-.5,1) .. controls +(-.5,0) and +(1.5,0) .. (-3.5,-2.5);

\draw[white, fill] (.6,1.85) rectangle (-.65,2.85);
\draw[] (.6,1.85) rectangle (-.6,2.85);
\node at (0,2.35) {\small{$m$-$1$}};

\pgfsetlinewidth{5*\pgflinewidth}
\draw[white] (3.5,1).. controls +(-1.5,0) and +(.5,0) .. (.5,-2.5) .. controls +(-.5,0) and +(.5,0) .. (-.5,-1) .. controls +(-.5,0) and +(1.5,0) .. (-3.5,2.5);
\pgfsetlinewidth{.2*\pgflinewidth}
\draw (3.5,1).. controls +(-1.5,0) and +(.5,0) .. (.5,-2.5) .. controls +(-.5,0) and +(.5,0) .. (-.5,-1) .. controls +(-.5,0) and +(1.5,0) .. (-3.5,2.5);
\pgfsetlinewidth{5*\pgflinewidth}
\draw[white] (3.5,2.5) .. controls +(-1.5,0) and +(.5,0) ..  (.5,-1) .. controls +(-.5,0) and +(.5,0) .. (-.5,-2.5) .. controls +(-.5,0) and +(1.5,0) .. (-3.5,1);
\pgfsetlinewidth{.2*\pgflinewidth}
\draw (3.5,2.5) .. controls +(-1.5,0) and +(.5,0) ..  (.5,-1) .. controls +(-.5,0) and +(.5,0) .. (-.5,-2.5) .. controls +(-.5,0) and +(1.5,0) .. (-3.5,1);

\draw[white, fill] (.65,-1.25) rectangle (-.65,-2.25);
\draw[] (.65,-1.25) rectangle (-.65,-2.25);
\node at (0,-1.75) {$k$};
\end{tikzpicture}
   \caption{\hspace{4pt}}\label{fig:S4ind1A}
    \end{subfigure} %
~
\begin{subfigure}[t]{0.49\textwidth}
        \centering
\begin{tikzpicture}[thick, scale =.65]

\draw[red, very thick, dashed] (0,3.5)--(0,-3);

\draw (3.5,-2.5)  .. controls +(-1.5,0) and +(.5,0) .. (.5,1) .. controls +(-.5,0) and +(.5,0) .. (-.5,1.5).. controls +(-.25,0)  and +(.125,-.125) .. (-2.5,-1.25) .. controls +(-.125,.125) and +(-.25,0) .. (-.5,1.75) .. controls +(.5,0) and +(.5,0) .. (-.5,3) .. controls +(-.5,0) and +(1.5,0) .. (-3.5,-1);

\pgfsetlinewidth{5*\pgflinewidth}
\draw[white] (3.5,-1) .. controls +(-1.5,0) and +(.5,0) .. (.5,3)  .. controls +(-.5,0) and +(-.5,0) .. (.5,1.75) .. controls +(.25,0) and +(.25,0) .. (.5,1.5)  .. controls +(-.5,0) and +(.5,0) .. (-.5,1) .. controls +(-.5,0) and +(1.5,0) .. (-3.5,-2.5);
\pgfsetlinewidth{.2*\pgflinewidth}

\draw (3.5,-1) .. controls +(-1.5,0) and +(.5,0) .. (.5,3)  .. controls +(-.5,0) and +(-.5,0) .. (.5,1.75) .. controls +(.25,0) and +(.125,.125) .. (2.5,-1.25) .. controls +(-.125,-.125) and +(.25,0) .. (.5,1.5)  .. controls +(-.5,0) and +(.5,0) .. (-.5,1) .. controls +(-.5,0) and +(1.5,0) .. (-3.5,-2.5);

\draw[white, fill] (.6,1.85) rectangle (-.65,2.85);
\draw[] (.6,1.85) rectangle (-.6,2.85);
\node at (0,2.35) {\small{$m$-$1$}};

\pgfsetlinewidth{5*\pgflinewidth}
\draw[white] (3.5,1).. controls +(-1.5,0) and +(.5,0) .. (.5,-2.5) .. controls +(-.5,0) and +(.5,0) .. (-.5,-1) .. controls +(-.5,0) and +(1.5,0) .. (-3.5,2.5);
\pgfsetlinewidth{.2*\pgflinewidth}
\draw (3.5,1).. controls +(-1.5,0) and +(.5,0) .. (.5,-2.5) .. controls +(-.5,0) and +(.5,0) .. (-.5,-1) .. controls +(-.5,0) and +(1.5,0) .. (-3.5,2.5);
\pgfsetlinewidth{5*\pgflinewidth}
\draw[white] (3.5,2.5) .. controls +(-1.5,0) and +(.5,0) ..  (.5,-1) .. controls +(-.5,0) and +(.5,0) .. (-.5,-2.5) .. controls +(-.5,0) and +(1.5,0) .. (-3.5,1);
\pgfsetlinewidth{.2*\pgflinewidth}
\draw (3.5,2.5) .. controls +(-1.5,0) and +(.5,0) ..  (.5,-1) .. controls +(-.5,0) and +(.5,0) .. (-.5,-2.5) .. controls +(-.5,0) and +(1.5,0) .. (-3.5,1);
\draw[white, fill] (.6,-1.25) rectangle (-.6,-2.25);
\draw[] (.6,-1.25) rectangle (-.6,-2.25);
\node at (0,-1.75) {$k$};

\end{tikzpicture}
   \caption{\hspace{4pt}}\label{fig:S4ind1B}
    \end{subfigure} %

\begin{subfigure}[t]{0.49\textwidth}
        \centering
\begin{tikzpicture}[thick, scale =.65]

\draw[red, very thick, dashed] (0,3.5)--(0,-3);

\draw (3.5,-2.5)  .. controls +(-1.5,0) and +(.5,0) .. (.5,-2.25) .. controls +(-.5,0) and +(.5,0) .. (-.5,-1.75).. controls +(-.25,0)  and +(.125,-.125) .. (-2.5,-1.25) .. controls +(-.125,.125) and +(-.25,0) .. (-.5,1.75) .. controls +(.5,0) and +(.5,0) .. (-.5,3) .. controls +(-.5,0) and +(1.5,0) .. (-3.5,-1);

\pgfsetlinewidth{5*\pgflinewidth}
\draw[white]  (3.5,-1) .. controls +(-1.5,0) and +(.5,0) .. (.5,3)  .. controls +(-.5,0) and +(-.5,0) .. (.5,1.75) .. controls +(.25,0) and +(.125,.125) .. (2.5,-1.25) .. controls +(-.125,-.125) and +(.25,0) .. (.5,-1.75) .. controls +(-.5,0) and +(.5,0) .. (-.5,-2.25)  .. controls +(-.5,0) and +(1.5,0) .. (-3.5,-2.5);
\pgfsetlinewidth{.2*\pgflinewidth}

\draw (3.5,-1) .. controls +(-1.5,0) and +(.5,0) .. (.5,3)  .. controls +(-.5,0) and +(-.5,0) .. (.5,1.75) .. controls +(.25,0) and +(.125,.125) .. (2.5,-1.25) .. controls +(-.125,-.125) and +(.25,0) .. (.5,-1.75) .. controls +(-.5,0) and +(.5,0) .. (-.5,-2.25)  .. controls +(-.5,0) and +(1.5,0) .. (-3.5,-2.5);

\draw[white, fill] (.6,1.85) rectangle (-.65,2.85);
\draw[] (.6,1.85) rectangle (-.6,2.85);
\node at (0,2.35) {\small{$m$-$1$}};
\pgfsetlinewidth{5*\pgflinewidth}
\draw[white] (3.5,1).. controls +(-1.5,0) and +(.5,0) .. (.5,.35)  .. controls +(-.5,0) and +(.5,0) .. (-.5,1.5) .. controls +(-.5,0) and +(1.5,0) .. (-3.5,2.5);
\pgfsetlinewidth{.2*\pgflinewidth}
\draw (3.5,1).. controls +(-1.5,0) and +(.5,0) .. (.5,.25)  .. controls +(-.5,0) and +(.5,0) .. (-.5,1.55) .. controls +(-.5,0) and +(1.5,0) .. (-3.5,2.5);
\pgfsetlinewidth{5*\pgflinewidth}
\draw[white] (3.5,2.5) .. controls +(-1.5,0) and +(.5,0) .. (.5,1.5) .. controls +(-.5,0) and +(.5,0) .. (-.5,.25)  .. controls +(-.5,0) and +(1.5,0) .. (-3.5,1);
\pgfsetlinewidth{.2*\pgflinewidth}
\draw (3.5,2.5) .. controls +(-1.5,0) and +(.5,0) .. (.5,1.55) .. controls +(-.5,0) and +(.5,0) .. (-.5,.25)  .. controls +(-.5,0) and +(1.5,0) .. (-3.5,1);

\draw[white, fill] (.6,.4) rectangle (-.6,1.4);
\draw[] (.6,.4) rectangle (-.6,1.4);
\node at (0,.9) {$k$};
\end{tikzpicture}
 \caption{\hspace{4pt}}\label{fig:S4ind1C}
\end{subfigure}%
~%
\begin{subfigure}[t]{0.49\textwidth}
\centering
\begin{tikzpicture}[thick, scale =.65]

\draw[red, very thick, dashed] (0,3)--(0,-3);

\draw (3.5,2.5) -- (.5,2.5) .. controls +(-.5,0) and +(.5,0) .. (-.5,1)--(-3.5,1);
\draw (3.5,1) --  (.5,1) .. controls +(-.5,0) and +(.5,0) .. (-.5,2.5)--(-3.5,2.5);
\draw[white, fill] (.6,1.25) rectangle (-.6,2.25);
\draw[] (.6,1.25) rectangle (-.6,2.25);
\node at (0,1.75) {$k$};

\draw (3.5,-2.5) -- (.5,-2.5) .. controls +(-.5,0) and +(.5,0) .. (-.5,-1)--(-3.5,-1);
\draw (3.5,-1) --  (.5,-1) .. controls +(-.5,0) and +(.5,0) .. (-.5,-2.5)--(-3.5,-2.5);
\draw[white, fill] (.6,-1.25) rectangle (-.6,-2.25);
\draw[] (.6,-1.25) rectangle (-.6,-2.25);
\node at (0,-1.75) {$m$};
\end{tikzpicture}
 \caption{\hspace{4pt}}\label{fig:S4ind1D}
\end{subfigure}
    \caption{Decomposition of an $S4(m, k)$-move.}\label{fig:S4ind1}
\end{figure}
More precisely, to go from Figure~\ref{fig:S4ind1A} to Figure~\ref{fig:S4ind1B} we use two symmetric second Reidemeister moves off the axis, to go from Figure~\ref{fig:S4ind1B} to Figure~\ref{fig:S4ind1C} one S4$(1,k)$-move and to go from Figure~\ref{fig:S4ind1C} to Figure~\ref{fig:S4ind1D} one $S4(m-1,k)$-move. A similar sequence of moves can be used to prove the inductive step for an $S4(h, n)$-move. 

For the other choices of signs of $m$ and $n$ the argument is essentially the same, except that one needs to perform the double induction on $|m|$ and $|n|$ and modify accordingly Figure~\ref{fig:S4ind1} and its analogue for the $S4(h, n)$-move. The obvious details are left to the reader.
\end{proof}

\begin{lemma}\label{lemma:S2n}
Suppose that the symmetric diagrams $D$ and $D^\prime$ differ by the~$S2(\pm, n)$-move defined in Figure \ref{fig:S2generale}. Then, $D$ and $D^\prime$ are connected by a sequence of symmetric  Reidemeister moves off the axis, $S4$-moves and $\vert n \vert$ $S2$-moves. 
In particular, $D$ and $D^\prime$ are symmetrically equivalent. 
\end{lemma}
\begin{figure}[ht]
\begin{tikzpicture}[thick, scale =.85]

\draw[red,dashed] (0,2) -- (0,-2) ;

\draw (-2,-2) .. controls +(.5,.5) and +(-1.5,1.5) .. (-0.25,.5);
\draw (2,-2) .. controls +(-.5,.5) and +(1.5,1.5) .. (0.25,.5);

\pgfsetlinewidth{5*\pgflinewidth}
\draw[white] (-2,2) .. controls +(.5,-.5) and +(-1.5,-1.5) .. (-0.25,-.5);
\draw[white] (2,2) .. controls +(-.5,-.5) and +(1.5,-1.5) .. (0.25,-.5);
\pgfsetlinewidth{.2*\pgflinewidth}

\draw (-2,2) .. controls +(.5,-.5) and +(-1.5,-1.5) .. (-0.25,-.5);
\draw (2,2) .. controls +(-.5,-.5) and +(1.5,-1.5) .. (0.25,-.5);
\draw[fill, white] (-0.4,.5) rectangle (0.4,-.5);
\draw (-0.4,.5) rectangle (0.4,-.5);

\node at (0,0) {$n$};

\node at (4,.5) {$S2(-,n)$};
\draw[latex-latex] (2.5,0) -- (5.5,0);

\begin{scope}[shift={+(8,0)}]
\draw[red,dashed] (0,2) -- (0,-2) ;

\draw (-2,2) .. controls +(.5,-0.25) and +(0,1) .. (-0.25,.5);
\draw (2,2) .. controls +(-.5,-.25) and +(0,1) .. (0.25,.5);

\draw (-2,-2) .. controls +(.5,0.25) and +(0,-1) .. (-0.25,-.5);
\draw (2,-2) .. controls +(-.5,.25) and +(0,-1) .. (0.25,-.5);
\draw[fill, white] (-0.4,.5) rectangle (0.4,-.5);
\draw (-0.4,.5) rectangle (0.4,-.5);

\node at (0,0) {$n$};

\end{scope}
\end{tikzpicture}
\caption{Definition of the $S2(-,n)$-move. The $S2(+,n)$-move is obtained by switching the crossings off the axis in the above picture.}
\label{fig:S2generale}
\end{figure}
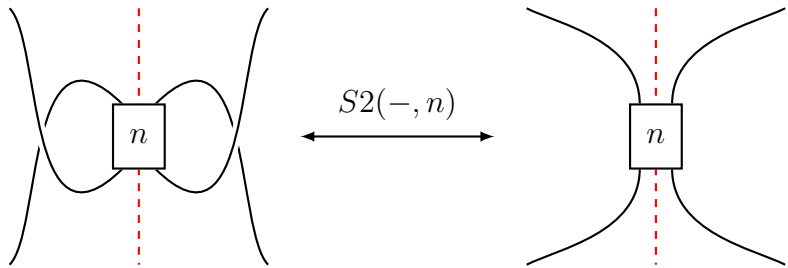

\begin{proof}
Note that the statement is obvious for $n=0,1,-1$. We describe the proof for the $S(-,n)$-move with 
$n<0$ because the other cases can be proved similarly. We are going to argue by induction on $n$, so we 
start assuming that the statement is true for $S(-,k)$-moves with $n < k \leq -1$. 
Performing a symmetric $R2$-move on the left-hand side tangle of Figure~\ref{fig:S2generale} we obtain the tangle of Figure~\ref{fig:S2A}. After an $S4(-1,n+1)$-move and some symmetric Reidemeister moves off the axis, the tangle of Figure~\ref{fig:S2A} can be modified into the tangle in Figure~\ref{fig:S2B}. By Lemma~\ref{lemma:S4mn} this means that the tangles of Figures~\ref{fig:S2A} and~\ref{fig:S2B} are obtained from each other via a sequence of $S4$-moves and Reidemeister moves off the axis. By a third Reidemeister move off the axis followed by a second Reidemeister move off the axis, the tangle of Figure~\ref{fig:S2B} can be turned into the tangle in Figure~\ref{fig:S2C}. Now we make use of the inductive hypothesis and perform an $S(-,n+1)$-move to get the tangle of Figure~\ref{fig:S2D}. Finally, a single $S2$-move leads us from Figure~\ref{fig:S2D} to the right-hand side of Figure~\ref{fig:S2generale}, concluding the proof.
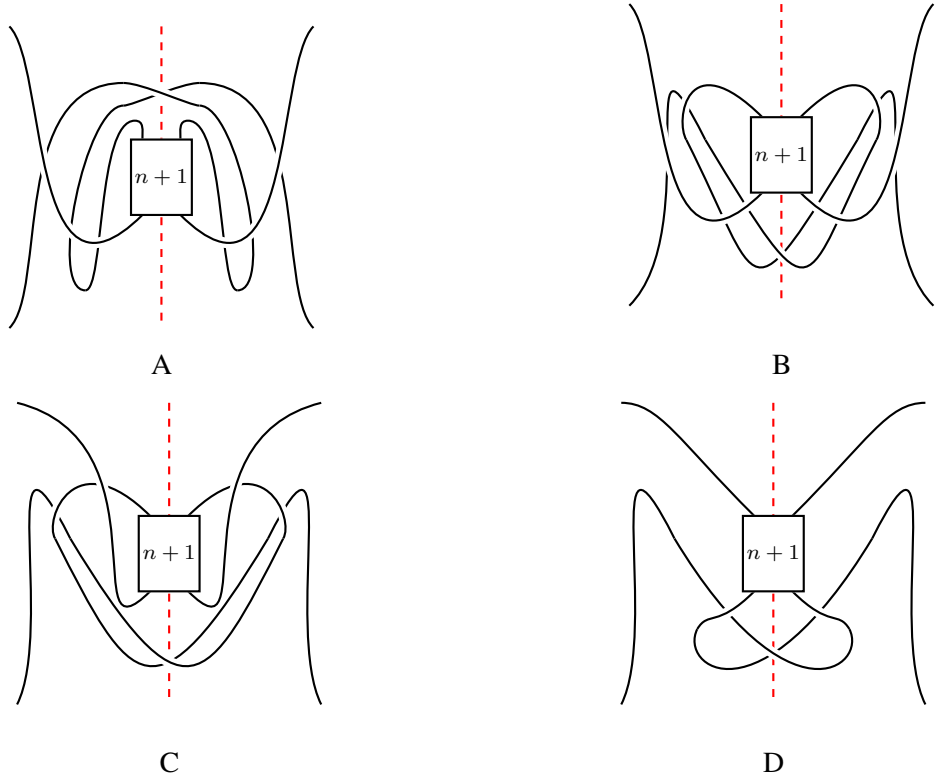
\begin{figure}[ht]
    \centering
    \begin{subfigure}[t]{0.49\textwidth}
    \centering
\begin{tikzpicture}[thick]

\draw[red,dashed] (0,2) -- (0,-2) ;

%
%
%
%
%
%

\draw (-1,-1.5) .. controls +(.25,0) and +(-.5,0) .. (-0.35,.75).. controls +(.125,0) and +(0,.125) ..(-0.25,.5);
\draw (1,-1.5) .. controls +(-.25,0) and +(.5,0) .. (0.35,.75) .. controls +(-.125,0) and +(0,.125) ..(0.25,.5);

\draw (-1,-1.5) .. controls +(-.5,0) and +(-.5,0) .. (-0.5,.95);
\draw (1,-1.5) .. controls +(.5,0) and +(.5,0) .. (0.5,.95);

\draw (.5,1.25) .. controls +(-.25,0) and +(.125,0) .. (-0.5,.95);
\pgfsetlinewidth{5*\pgflinewidth}
\draw[white] (-0.5,1.25) .. controls +(.25,0) and +(-.125,0) .. (0.5,.95);
\pgfsetlinewidth{.2*\pgflinewidth}
\draw(-0.5,1.25) .. controls +(.25,0) and +(-.125,0) .. (0.5,.95);

\draw (-2,-2) .. controls +(.5,.5) and +(-1.5,0) .. (-0.5,1.25);
\draw (2,-2) .. controls +(-.5,.5) and +(1.5,0) .. (0.5,1.25);

\pgfsetlinewidth{5*\pgflinewidth}
\draw[white] (-2,2) .. controls +(.5,-.5) and +(-1.5,-1.5) .. (-0.25,-.5);
\draw[white] (2,2) .. controls +(-.5,-.5) and +(1.5,-1.5) .. (0.25,-.5);
\pgfsetlinewidth{0.2*\pgflinewidth}

\draw (-2,2) .. controls +(.5,-.5) and +(-1.5,-1.5) .. (-0.25,-.5);
\draw (2,2) .. controls +(-.5,-.5) and +(1.5,-1.5) .. (0.25,-.5);
\draw[fill, white] (-0.4,.5) rectangle (0.4,-.5);
\draw (-0.4,.5) rectangle (0.4,-.5);

\node at (0,0) {\tiny $n+1$};
\end{tikzpicture}
    \caption{\hspace{4pt}}\label{fig:S2A}
    \end{subfigure} %
~
\begin{subfigure}[t]{0.49\textwidth}
        \centering
\begin{tikzpicture}[thick]

\draw[red,dashed] (0,2) -- (0,-2) ;

\draw (2,-2) .. controls +(-1,1) and +(1,2) .. (1,.2);
\draw (-2,-2) .. controls +(1,1) and +(-1,2) .. (-1,.2);

\draw (1,.2) .. controls +(-1.5,-2.5) and +(1,-2) ..  (-1.25,.2);

\pgfsetlinewidth{5*\pgflinewidth}
\draw[white]  (-1,.2) .. controls +(1.5,-2.5) and +(-1,-2) ..  (1.25,.2);
\pgfsetlinewidth{.2*\pgflinewidth}

\draw (-1,.2) .. controls +(1.5,-2.5) and +(-1,-2) ..  (1.25,.2);

\pgfsetlinewidth{5*\pgflinewidth}
\draw[white]  (-1.25,.2) .. controls +(-.15,.25) and +(-1,1) .. (-0.25,.5);
\draw[white] (1.25,.2) .. controls +(.15,.25) and +(1,1) .. (0.25,.5);
\pgfsetlinewidth{.2*\pgflinewidth}

\draw (-1.25,.2) .. controls +(-.15,.25) and +(-1,1) .. (-0.25,.5);
\draw (1.25,.2) .. controls +(.15,.25) and +(1,1) .. (0.25,.5);

\pgfsetlinewidth{5*\pgflinewidth}
\draw[white] (-2,2) .. controls +(.5,-.5) and +(-1.5,-1.5) .. (-0.25,-.5);
\draw[white] (2,2) .. controls +(-.5,-.5) and +(1.5,-1.5) .. (0.25,-.5);
\pgfsetlinewidth{.2*\pgflinewidth}

\draw (-2,2) .. controls +(.5,-.5) and +(-1.5,-1.5) .. (-0.25,-.5);
\draw (2,2) .. controls +(-.5,-.5) and +(1.5,-1.5) .. (0.25,-.5);
\draw[fill, white] (-0.4,.5) rectangle (0.4,-.5);
\draw (-0.4,.5) rectangle (0.4,-.5);

\node at (0,0) {\tiny $n+1$};
\end{tikzpicture}
   \caption{\hspace{4pt}}\label{fig:S2B}
    \end{subfigure} %

\begin{subfigure}[b]{0.49\textwidth}
        \centering
\begin{tikzpicture}[thick]

\draw[red,dashed] (0,2) -- (0,-2) ;

\draw (2,-2) .. controls +(-.5,1) and +(1,2) .. (1.3,.2);
\draw (-2,-2) .. controls +(.5,1) and +(-1,2) .. (-1.3,.2);

\draw (1.3,.2) .. controls +(-1.5,-2.5) and +(1,-2) ..  (-1.5,.2);

\pgfsetlinewidth{5*\pgflinewidth}
\draw[white]  (-1.3,.2) .. controls +(1.5,-2.5) and +(-1,-2) ..  (1.5,.2);
\pgfsetlinewidth{.2*\pgflinewidth}

\draw (-1.3,.2) .. controls +(1.5,-2.5) and +(-1,-2) ..  (1.5,.2);

\pgfsetlinewidth{5*\pgflinewidth}
\draw[white]  (-1.5,.2) .. controls +(-.15,.25) and +(-1,1) .. (-0.25,.5);
\draw[white] (1.5,.2) .. controls +(.15,.25) and +(1,1) .. (0.25,.5);
\pgfsetlinewidth{.2*\pgflinewidth}

\draw (-1.5,.2) .. controls +(-.15,.25) and +(-1,1) .. (-0.25,.5);
\draw (1.5,.2) .. controls +(.15,.25) and +(1,1) .. (0.25,.5);

\pgfsetlinewidth{5*\pgflinewidth}
\draw[white] (-2,2) .. controls +(2,-.5) and +(-1,-1) .. (-0.25,-.5);
\draw[white] (2,2) .. controls +(-2,-.5) and +(1,-1) .. (0.25,-.5);
\pgfsetlinewidth{.2*\pgflinewidth}

\draw (-2,2) .. controls +(2,-.5) and +(-1,-1) .. (-0.25,-.5);
\draw (2,2) .. controls +(-2,-.5) and +(1,-1) .. (0.25,-.5);

\draw[fill, white] (-0.4,.5) rectangle (0.4,-.5);
\draw (-0.4,.5) rectangle (0.4,-.5);

\node at (0,0) {\tiny $n+1$};
\end{tikzpicture}
 \caption{\hspace{4pt}}\label{fig:S2C}
\end{subfigure}%
~%
\begin{subfigure}[b]{0.49\textwidth}
\centering
\begin{tikzpicture}[thick]

\draw[red,dashed] (0,2) -- (0,-2) ;
\draw (2,-2) .. controls +(-.5,1) and +(1,2) .. (1.3,.2);
\draw (-2,-2) .. controls +(.5,1) and +(-1,2) .. (-1.3,.2);

\draw (1.3,.2) .. controls +(-1.5,-2.5) and +(-.25,-.5) .. (-1,-1);

\pgfsetlinewidth{5*\pgflinewidth}
\draw[white]   (-1.3,.2) .. controls +(1.5,-2.5) and +(.25,-.5) .. (1,-1);
\pgfsetlinewidth{.2*\pgflinewidth}

\draw (-1.3,.2) .. controls +(1.5,-2.5) and +(.25,-.5) .. (1,-1);

\pgfsetlinewidth{5*\pgflinewidth}
\draw[white] (-1,-1) .. controls +(.15,.25) and +(-.5,-.5) .. (-0.25,-.5);
\draw[white] (1,-1) .. controls +(-.15,.25) and +(.5,-.5) .. (0.25,-.5);
\pgfsetlinewidth{.2*\pgflinewidth}

\draw (-1,-1) .. controls +(.15,.25) and +(-.5,-.5) .. (-0.25,-.5);
\draw (1,-1) .. controls +(-.15,.25) and +(.5,-.5) .. (0.25,-.5);

\draw (-2,2) .. controls +(.5,0) and +(-1,1) .. (-0.25,.5);
\draw (2,2) .. controls +(-.5,0) and +(1,1) .. (0.25,.5);

\draw[fill, white] (-0.4,.5) rectangle (0.4,-.5);
\draw (-0.4,.5) rectangle (0.4,-.5);

\node at (0,0) {\tiny $n+1$};
\end{tikzpicture}
 \caption{\hspace{4pt}}\label{fig:S2D}
\end{subfigure}
    \caption{Decomposition of an $S2(\pm, n)$-move.}\label{fig:S2}
\end{figure}
\end{proof}
\begin{proof}[Proof of Theorem~\ref{t:invariance}]
We will argue that, whenever the diagrams $D$ and $D^\prime$ differ by a symmetric Reidemeister move $M$, then $D(n)$ and $D^\prime(n)$ are weakly symmetrically equivalent
if $M$ is of type $S2(v)$, and symmetrically equivalent otherwise. This is clear for symmetric Reidemeister moves off the axis and $S2(h)$-moves because they do not involve crossings on the axis. Suppose that $D^\prime$ is obtained from $D$ by applying an $S1$-move. Then, it is immediate that $D^\prime(n)$ is obtained from $D(n)$ by applying $\vert n\vert$ $S1$-moves. A similar reasoning applies to $S2(v)$-moves and $S3$-moves, and since all the verifications are very simple we leave them to the reader. If $D^\prime$ is obtained from $D$ by an $S4(\ep_1,\ep_2)$-move with $\ep_i\in\{\pm\}$, then $D^\prime(n)$ is obtained from $D(n)$ via an $S4(\epsilon_1 n, \epsilon_2 n)$-move and by Lemma~\ref{lemma:S4mn} $D(n)$ and $D^\prime(n)$ are symmetrically equivalent. Finally, if $D^\prime$ is obtained from $D$ by an $S2(\pm)$-move, then $D^\prime(n)$ is obtained from $D(n)$ via a $S2(\pm, n)$-move and by~Lemma~\ref{lemma:S2n} $D(n)$ and $D'(n)$ are symmetrically equivalent. 
\end{proof}

\section{Negative results for $W$ and the Potts-refined spin model invariants}\label{s:negative}
Our aim in this section is to show that the diagrams $D_4(h)$ and $D'_4(h)$ cannot be distinguished 
up to any symmetric equivalence using neither Eisermann and Lamm's refined Jones polynomial nor any invariant  coming from a Potts-refined topological spin model. This will be established in Corollary~\ref{c:negative}. 
We start with the following Proposition~\ref{p:W-equality}, which will be used in the proof of Corollary~\ref{c:negative}. Recall that $\vec{\D}$ denotes the set of oriented planar link diagrams $D\subset\bR^2$ transverse to the axis $\ell =\{0\}\x\bR$. Let $U^m$, for $m\geq 1$, denote any crossingless, symmetric union diagram of the $m$-component unlink. 
\begin{prop}\label{p:W-equality}
Let $R$ be a ring and $\Psi\co\vec{\D}\to R$ a map such that
\begin{enumerate}
\item 
$\Psi(D)=\Psi(D')$ if $D$ and $D'$ are weakly symmetrically equivalent;
\item 
$\Psi\left(\rightcrossaxis{}\right) = a_+ \Psi\left(\horresaxis{}\right) + b_+ \Psi\left(\vertresaxis\right)$ and 
$\Psi \left(\leftcrossaxis{}\right) = a_- \Psi\left(\horresaxis{}\right) + b_- \Psi\left(\vertresaxis\right)$, with 
\item 
$b_+ b_- = 1$ and $(a_+ b_- + a_- b_+) \Psi(U^2) + a_+ a_- \Psi(U^3) =0$.
\end{enumerate}
Then, $\Psi(D_4(h)) = \Psi(D'_4(h))$ for each $h\in\bZ$.
\end{prop} 

\begin{proof}
Clearly $D_4(0)=D^\prime_4(0)$. We only prove the statement for $h>0$ because the proof for $h<0$ is essentially the same. Let $D_4(t,h,b)$, $t,b\geq 0$, $h\geq 1$, be the diagram obtained from $D_4$ by replacing the top (respectively bottom) crossing on the axis with $t$ (respectively $b$) consecutive crossings on the axis of the same sign, and each of the other crossings on the axis with $h$ consecutive crossings on the axis of the same sign. Observe that $D_4(h,h,h) = D_4(h)$ and $D_4(0,h,0) = D'_4(h)$ for each $h\geq 1$. Therefore, it suffices to prove that $\Psi(D_4(k,h,k)) = \Psi(D_4(h))$ for each $h\geq k\geq 0$ with $h\geq 1$. This follows by an easy downward induction on $k$ starting from $k=h\geq 1$ once we show that the equality $\Psi(D_4(k,h,k)) = \Psi(D_4(k-1,h,k-1))$ holds. 
It will be convenient to use the following terminology and notation. We call a horizontal resolution of a crossing on the axis a {\em $0$-resolution} and a vertical resolution of such a crossing a {\em $1$-resolution}. We denote by $D_{x y}$, with $x,y\in\{0,1\}$, any symmetric union diagram obtained from $D_4(k,h,k)$ by an $x$-resolution of any of its $k$ top crossings on the axis and a $y$-resolution of any of its $k$ bottom crossings on the axis. It is easy to check that $D_{0 1}$ and $D_{1 0}$ are weakly symmetrically equivalent to $U^2$, $D_{0 0}$ is weakly symmetrically equivalent to $U^3$ and $D_{1 1} = D_4(k-1,h,k-1)$. A simple calculation using Assumption~(2) yields 
\[
\Psi(D_4(k,h,k)) = b_+ b_- \Psi(D_4(k-1,h,k-1)) + (a_+ b_- + a_- b_+) \Psi(U^2)  + a_+ a_- \Psi(U^3),
\]
which by (3) gives the claimed equality $\Psi(D_4(k,h,k)) = \Psi(D_4(k-1,h,k-1))$. 
\end{proof}

\begin{cor}\label{c:negative}
For any $h\in\bZ$, the refined Jones polynomial and any Potts-refined spin model invariant take the same values on $D_4(h)$ and $D'_4(h)$. 
\end{cor} 

\begin{proof} 
Let $W$ be the refined Jones polynomial from~\cite{Ei.La11}. We recalled in Subsection~\ref{ss:rJp} that $W$ satisfies Assumption~$(1)$ of Proposition~\ref{p:W-equality}. By Equation~\eqref{e:nocrossings}, since 
$V_{U^m}(t) = (-t^{1/2}-t^{-1/2})^{m-1}$  we obtain $W(U^m) = (-s^{1/2}-s^{-1/2})^{m-1}$. Together with Equations~\eqref{e:recursionright} and~\eqref{e:recursionleft} this immediately implies that $W$ satisfies Assumptions~$(2)$ and~$(3)$ of Proposition~\ref{p:W-equality} and therefore $W(D_4(h)) = W(D'_4(h))$ 
for every $h\in\bZ$. 

Let $\widehat{M} = (W^+, V^+, d)$ be a Potts-refined spin model. By Equation~\ref{e:se-invariance} we know that $I_{\widehat{M}}$ satisfies Assumption~$(1)$ of Proposition~\ref{p:W-equality}, and it follows immediately from the definition that $I_{\widehat{M}}(U^m) = d^m$. Using that $V^+ =  (-\xi^{-3} - \xi) I + \xi J = d\xi^{-1} I + \xi J$, $V^- =  -\xi^3 I + \xi^{-1} (J-I) = (-\xi^3-\xi^{-1}) I + \xi^{-1} J = d\xi I + \xi^{-1} J$ and $d=-\xi^{-2}-\xi^2$, it is straightforward to check that 
\[
I_{\widehat{M}} \left(\rightcrossaxis{}\right) = -\xi^{-2} I_{\widehat{M}} \left(\horresaxis{}\right) - \xi^{-4} I_{\widehat{M}} \left(\vertresaxis
\right)\ \text{and} \ 
I_{\widehat{M}} \left(\leftcrossaxis{}\right) = -\xi^2 I_{\widehat{M}} \left(\horresaxis{}\right) - \xi^4 I_{\widehat{M}} \left(\vertresaxis\right).
\]
Thus, $I_{\widehat{M}}$ satisfies Assumptions~$(2)$ and~$(3)$ of Proposition~\ref{p:W-equality} and $I_{\widehat{M}} (D_4(h)) = I_{\widehat{M}}(D'_4(h))$ for every $h\in\bZ$.
\end{proof}

\section{Proof of Theorem~\ref{t:main}}\label{s:applications} 
In this section we show that the diagrams $D_4(2)$ and $D'_4(2)$ represent distinct knots $K$ and $K'$. Since this implies that $D_4(2)$ and $D'_4(2)$ are not Reidemeister equivalent, combining this fact with Theorem~\ref{t:invariance} yields Theorem~\ref{t:main}.

As one referee pointed out to us, to show that $K$ and $K'$ are distinct it is possible to use both the Kauffmann polynomial and the colored Jones polynomial. Another possibility is to show that $K$ and $K'$ have different second Alexander ideals. In a previous version of this paper~\footnote{https://arxiv.org/abs/1901.10270v2} we worked out the details of the computation of the second Alexander ideals for the infinite families of knots $\{K_s\}$ and $\{K'_s\}$ given by the diagrams $D_{4s}(2)$ and $D'_{4s}(2)$, $s\geq 1$. As it turns out, the second Alexander ideals distinguish $K_s$ from $K'_s$ for each $s\geq 1$. In this paper we just prove that $K$ and $K'$ are distinct using their third cyclic branched covers.

\begin{lemma}\label{l:tnumbs}
Let $\Si = \Si_3(K)$ and $\Si' = \Si_3(K')$ denote the three-fold branched covers of $K$ and $K'$, respectively. Then, we have the following isomorphisms of Abelian groups:
\[
H_1(\Si;\bZ)\cong\bZ/7\bZ\oplus\bZ/7\bZ\oplus\bZ/7\bZ\oplus\bZ/7\bZ
\quad\text{and}\quad H_1(\Si';\bZ)\cong\bZ/49\bZ\oplus\bZ/49\bZ.
\]
\end{lemma} 

\begin{proof} 
We start by computing Seifert matrices for $K$ and $K'$. Consider the Seifert surface $\Sigma$ for $D_4(2)$ and the basis for its first homology group shown in Figure \ref{figure:basis_car}. The generators are divided into two groups, each of which is shown separately in Figure~\ref{figure:basis_car} to maximize readability. 
\begin{figure}[ht]
	\centering
	\input{GenDprime4_full}
	\caption{A Seifert surface $\Si$ for $D_4(2)$ and a basis of $H_1(\Si;\bZ)$}
	\label{figure:basis_car}
\end{figure}
It is straightforward to check that the associated Seifert matrix is 
\[
V = \sm{0 & 1 & 0 & 0 & 0 & 0 & 0 & 0 & 0 & -1 \\
0 & -1 & 0 & 0 & 0 & 0 & 0 & 0 & 0 & 0 \\
0 & 1 & 0 & 0 & 0 & 0 & 0 & 0 & 0 & 0 \\
0 & 0 & 1 & 1 & 1 & 0 & 0 & 0 & 0 & 0 \\
0 & 0 & 0 & 0 & 0 & 1 & 0 & 0 & 0 & -1 \\
0 & 0 & 0 & 0 & 0 & -1 & 0 & 0 & 0 & 0 \\
0 & 0 & 0 & 0 & 0 & 1 & 0 & 0 & 0 & 0 \\
0 & 0 & 0 & 0 & 0 & 0 & 1 & 2 & 1 & 0 \\
0 & 0 & 0 & 0 & 0 & 0 & 1 & 1 & 0 & 1 \\
0 & 0 & 1 & 0 & 0 & 0 & 1 & 0 & 0 & 2}.
\]
Next, we consider a Seifert surface  $\Si'$ for $D'_4(2)$ and the basis for 
$H_1(\Si',\bZ)$ illustrated in Figure \ref{figure:basis_ors}. As before, the generators are divided into two groups, which are shown separately.
\begin{figure}[ht]
	\centering
	\input{GenD4_full}
	\caption{A Seifert surface $\Si'$ for $D'_4(2)$ and a basis of $H_1(\Si',\bZ)$}
	\label{figure:basis_ors}
\end{figure}
It is easy to check that the corresponding Seifert matrix $V'$ is given by
\[
V' = \sm{
0 & 0 & 0 & 0 & 0 & 0 & 0 & 1 \\
1 & 1 & -1 & 0 & 0 & 0 & 0 & 0 \\
0 & 0 & 0 & 1 & 0 & 0 & 0 & 1 \\
0 & 0 & 0 & -1 & 0 & 0 & 0 & 0 \\
0 & 0 & 0 & -1 & 0 & 0 & 0 & 0 \\
0 & 0 & 0 & 0 & -1 & 1 & 0 & 0 \\
0 & 0 & 0 & 0 & 1 & 0 & -1 & 1 \\
1 & 0 & 0 & 0 & 1 & 0 & 0 & 2
}
\]
Given a Seifert matrix $V$ for a knot $K$, a presentation matrix for $H_1(\Si_3(K);\bZ)$ is given by $\Ga_V^3-(\Ga_V-I)^3$, where $\Ga_V = -V({}^t V - V)^{-1}$ and $I$ is the identity matrix~\cite[Satz I]{Se35} (see also~\cite[Theorem~3]{Tr62}). The statement follows by computing the elementary divisors of the matrices $\Ga_V$ and $\Ga_{V'}$. 
\end{proof}

\begin{proof}[Proof of Theorem~\ref{t:main}]
The statement is an immediate consequence of Lemma~\ref{l:tnumbs} and Theorem~\ref{t:invariance}.
\end{proof}

\bibliographystyle{abbrv}
\bibliography{biblio}

\begin{thebibliography}{1}

\bibitem{CL}
C.~Collari and P.~Lisca.
\newblock Symmetric union diagrams and refined spin models.
\newblock {\em Canadian Mathematical Bulletin}, 2018.
\newblock https://doi.org/10.4153/S0008439518000115 and ArXiv preprint
  1804.09157.

\bibitem{Ei.La11}
M.~Eisermann and C.~Lamm.
\newblock A refined {{Jones}} polynomial for symmetric unions.
\newblock {\em Osaka Journal of Mathematics}, 48(2):333--370, 2011.

\bibitem{Jo89}
V.~Jones.
\newblock On knot invariants related to some statistical mechanical models.
\newblock {\em Pacific Journal of Mathematics}, 137(2):311--334, 1989.

\bibitem{Se35}
H.~Seifert.
\newblock {\"U}ber das geschlecht von knoten.
\newblock {\em Mathematische Annalen}, 110:571--592, 1935.

\bibitem{Tr62}
H.~Trotter.
\newblock Homology of group systems with applications to knot theory.
\newblock {\em Annals of Mathematics}, 76:464--498, 1962.

\end{thebibliography}
\end{document}